\documentclass[aop,MSNbibl,citesort,dvips]{arximspdf}
\usepackage{mathbh}

\doi{10.1214/10-AOP630}
\volume{40}
\issue{2}
\pubyear{2012}
\firstpage{695}
\lastpage{713}

\makeatletter

\newcommand{\varliminf}{\mathop{\underline{\lim}}}
\newcommand{\varlimsup}{\mathop{\overline{\lim}}}

\newcommand{\N}{\mathbb N}
\newcommand{\RR}{\mathbb R}
\newcommand{\Z}{\mathbb Z}
\newcommand{\cB}{\mathcal{B}}
\newcommand{\cC}{\mathcal{C}}
\newcommand{\cD}{\mathcal{D}}
\newcommand{\cF}{\mathcal{F}}
\newcommand{\cH}{\mathcal{H}}
\newcommand{\cL}{\mathcal{L}}
\newcommand{\cS}{\mathcal{S}}
\newcommand{\cV}{\mathcal{V}}

\newcommand{\vareps}{\varepsilon}

\newcommand{\sgn}{\operatorname{sgn}}
\newcommand{\one}{\mathbh1}

\newtheorem{theorem}{Theorem}
\newtheorem{lemma}{Lemma}
\newtheorem{corollary}{Corollary}
\newtheorem{conjecture}{Conjecture}

\newproclaim{remarks}{Remarks}
\newproclaim{remark}{Remark}

\newcommand{\wick}[1]{:\! #1 \!:}

\makeatother

\setattribute{abstract}   {width}  {285pt}

\begin{document}
\begin{frontmatter}
\title{Diffusivity bounds for 1D Brownian polymers}
\runtitle{Diffusivity bounds for 1D Brownian polymers}

\begin{aug}
\author[A]{\fnms{Pierre} \snm{Tarr\`{e}s}\corref{}\thanksref{t1,t2}\ead[label=e1]{tarres@math.univ-toulouse.fr}},
\author[B]{\fnms{B\'{a}lint} \snm{T\'{o}th}\thanksref{t3}\ead[label=e2]{balint@math.bme.hu}} and
\author[C]{\fnms{Benedek} \snm{Valk\'{o}}\thanksref{t3,t4}\ead[label=e3]{valko@math.wisc.edu}}
\runauthor{P. Tarr\`{e}s, B. T\'{o}th and B. Valk\'{o}}
\affiliation{CNRS, Universit\'{e} de Toulouse, Budapest University of
Technology
and~University of Wisconsin}
\address[A]{P. Tarr\`{e}s\\
Institut de Math\'{e}matiques\\
CNRS, Universit\'{e} de Toulouse\\
118 route de Narbonne\\
31062 Toulouse Cedex 9\\
France\\
\printead{e1}} 
\address[B]{B. T\'{o}th\\
Institute of Mathematics\\
Budapest University of Technology\\
Egry J\'{o}zsef u.~1\\
Budapest, H-1111\\
Hungary\\
\printead{e2}}
\address[C]{B. Valk\'{o}\\
Department of Mathematics\\
University of Wisconsin Madison\\
480 Lincoln Drive\\
Madison, Wisconsin 53706\\
USA\\
\printead{e3}}
\end{aug}

\thankstext{t1}{On leave from the University of Oxford.}

\thankstext{t2}{Supported in part by a Leverhulme Prize.}

\thankstext{t3}{Supported in part by OTKA (Hungarian National
Research Fund) Grant K
60708.}

\thankstext{t4}{Supported in part by the NSF Grant DMS-09-05820.}

\received{\smonth{11} \syear{2009}}
\revised{\smonth{8} \syear{2010}}

%
\begin{abstract}
We study the asymptotic behavior of a self-interacting one-dimensio\-nal
Brownian polymer first introduced by Durrett and Rogers
[\textit{Probab. Theory Related Fields} \textbf{92} (1992) 337--349].
The polymer describes a~sto\-chastic process
with a drift which is a certain average of its local time.

We show that a smeared out version of the local time function as viewed
from the actual position of the process is a Markov process in a
suitably chosen function space, and that this process has a Gaussian
stationary measure. As a first consequence, this enables us to
partially prove a conjecture about the law of large numbers for the
end-to-end displacement of the polymer formulated in Durrett and Rogers
[\textit{Probab. Theory Related Fields} \textbf{92} (1992) 337--349].

Next we give upper and lower bounds for the variance of the process
under the stationary measure, in terms of the qualitative infrared
behavior of the interaction function. In particular, we show that in
the locally self-repelling case (when the process is essentially pushed
by the negative gradient of its own local time) the process is
super-diffusive.
\end{abstract}

%

\end{frontmatter}

\section{Introduction}
\label{s:intro}

\subsection{Historical background}
\label{ss:background}

Let $(X(t))_{t\ge0}$ be the random process defined by $X(0):=x_0\in
\RR$ and
%
\begin{equation}
\label{dr}
X(t)=
B(t)+\int_0^t \biggl(\xi(X(s))+\int_0^s f\bigl(X(s)-X(u)\bigr)\,du\biggr)\,ds,
\end{equation}
where $B(t)$ is a standard 1D Brownian motion, $f\dvtx\RR\to\RR$ is a
function with sufficient regularity, and $\xi\dvtx\RR\to\RR$ is an
initial drift profile with regularity (detailed below).

This process $X(t)$ was introduced by
Norris, Rogers and Williams \cite{norrisrogerswilliams87},
Durrett and Rogers \cite
{durrettrogers92}, as a model for the location of the end of a growing
polymer at time~$t$, in the case of zero initial profile ($\xi\equiv0$).

It is phenomenologically instructive to write the driving mechanism on
the right-hand side of (\ref{dr}) in terms of the occupation time
density (local time) of the process $X(t)$:
%
\begin{equation}
\label{drlt}
X(t)= B(t)+
\int_0^t\biggl\{\xi(X(s))+\int_{-\infty}^\infty
f(z)L\bigl(s,X(s)-z\bigr)\,dz\biggr\}\,ds,
\end{equation}
where
%
\begin{equation}
\label{lt}
L(s,y):=\partial_y
\int_0^s {\mathbh1}_{\{X(u)<y\}} \,du.
\end{equation}

Various choices of the function $f$ have been analyzed in detail and
mathematically deep, sometimes phenomenologically surprising results
have been obtained in the papers
\cite{cranstonlejan95,cranstonmountford96} and \cite
{mountfordtarres08}. For a detailed survey of the problem see \cite
{mountfordtarres08}.
However, satisfactory understanding of the asymptotic behavior of the
process (\ref{dr}) has not been reached in many interesting cases.

In particular, the following conjecture has remained open so far:
\begin{conjecture}[(Durrett and Rogers \cite{durrettrogers92})]
\label{conj:dur-rog}
Suppose $f$ has sufficient
fast decay at infinity, and
%
\begin{equation}
\label{sr}
f(-x)=-f(x)\quad\mbox{and}\quad
\sgn(f(x))=\sgn(x).
\end{equation}
Then $X(t)/t\to0$ a.s.
\end{conjecture}

T\'{o}th and Werner later conjectured that, under the same assumptions,
$X(t)/t^{2/3}$ converges in law, by analogy with the discrete
space--time self-repelling random walk on $\Z$ which displays this
$t^{2/3}$ asymptotic behavior (with identification of the limiting
distribution, see
\cite{toth95,toth99}), and with a~continuous space--time process arising as a
scaling limit constructed in~\cite{tothwerner98}. These studies were
stimulated by the so-called \textit{true self-avoiding random walk}
(TSAW) introduced in the physics literature by Amit, Parisi and Peliti~\cite{amitparisipeliti83}.

We partially prove Conjecture \ref{conj:dur-rog}, and obtain asymptotic
lower and upper bounds in the stationary regime which translate, in the
case (\ref{sr}), into
%
\begin{equation}
\varliminf_{t\to\infty} t^{-5/4} \mathbf{E}(X(t)^2)>0,\qquad
\varlimsup_{t\to\infty} t^{-3/2} \mathbf{E}(X(t)^2
)<\infty,
\end{equation}
where the lower bound is meant in the sense of Laplace transform (see
details later). We also show that the process $X(t)$ behaves
diffusively, for functions $f$ satisfying a certain summability
condition [see (\ref{sumcond})].\vadjust{\goodbreak} Our argument is based on the study of
an underlying Markov process living in the path space, which has
invariant Gaussian measure.

In the follow-up paper \cite{horvathtothveto09} the analogous polymer
model in dimensions $d\ge3$ is investigated. There full CLT is proved
for the locally self-repelling case in those dimensions, using the
nonreversible version of the Kipnis--Varadhan theory. As explained in
that paper, technical parts of that method do not apply (so far) in
lower dimensions.

\subsection{Assumptions on $f$}
\label{sec:assumptions}

We assume throughout the paper that the Brownian polymer processes
(\ref{dr}) are under the assumption that the function $f$ is the \textit
{negative gradient of an absolutely integrable smooth function of
positive type}, that is,
%
\begin{equation}
f(x)=-b^\prime(x),
\end{equation}
where
$b\in L^1(\RR)\cap C^{(\infty)}(\RR)$ has nonnegative Fourier
transform. Note that positive definiteness implies
%
\begin{equation}
b(-x)=b(x),\qquad
\sup_{x\in\RR} |b(x)|=b(0).
\end{equation}
Given that $b$ is of positive type,
it is actually sufficient to assume its infinitely differentiability
only at $x=0$. Indeed, since $b\in L^1({\mathbb R})$ and $\hat b(p)\ge0$, it
then follows that $\hat b$ has finite moments of all orders:
for all $k\in\N$,
%
\begin{equation}
\label{Fouriermoments}
\int_{-\infty}^\infty| {p} |^k \hat b(p) \,dp <\infty,
\end{equation}
and, hence, it follows that actually $b\in C^\infty(\RR)$.

Note that the regularity assumption is much more than really needed, we
assume it in order to make the technical arguments shorter.

\subsection{Underlying Markov process and invariant Gaussian measure}

First, we let $t\mapsto\zeta(t,x)$ be the ``drift function''
environment at time $t$ [i.e., $\zeta(t,x)$ is the drift that would be
endured by the particle at time $t$ if it were in $x$]:
%
\begin{equation}
\label{zetaeq}
\zeta(t,x)=\zeta(0,x)+\int_0^t b^\prime\bigl(X(s)-x\bigr)\,ds.
\end{equation}
The initial condition is $\zeta(0,x)=\xi(x)$ from (\ref{dr}) and
(\ref{drlt}). Then (\ref{dr}) reads
%
\begin{equation}
\label{Xeq}
X(t)=X(0) + B(t) + \int_0^t\zeta(s,X(s))\,ds.
\end{equation}

In other words,
%
\begin{equation}
\label{X_zeta_eqctd}
dX(t)= d B(t)+\zeta(t,X(t))\,dt,\qquad
d\zeta(t,x)= b^\prime\bigl(X(t)-x\bigr)\,dt.
\end{equation}

Now let $\eta$ be the environment profile as seen from the moving
point $X(t)$, that is,
%
\begin{equation}
\label{etadef}
x\mapsto\eta(t,x):=\zeta\bigl(t, X(t)+x\bigr).\vadjust{\goodbreak}
\end{equation}
Then $t\mapsto\eta(t):=\eta(t,\cdot)$ \textit{is a Markov process},
on the space of smooth functions of slow increase at infinity:
%
\begin{equation}
\label{Omega}
\Omega
:=
\{\omega\in C^{\infty}(\RR\to\RR) \dvtx
(\forall k\ge0,\forall l\ge1)\dvtx
\| {\omega} \|_{k,l}<\infty\},
\end{equation}
where $\| {\omega} \|_{k,l}$ are the seminorms
%
\begin{equation}
\label{seminorms}
\| {\omega} \|_{k,l} :=
\sup_{x\in\RR}
(1+| {x} |)^{-1/l}
\bigl| {\omega^{(k)}(x)} \bigr| ,\qquad
k\ge0,
l\ge1.
\end{equation}
$\Omega$ endowed with these seminorms $\| {\omega} \|
_{k,l}$, $k\ge
0$, $l\ge1$, is a Fr\'{e}chet space.

Note that the existence and uniqueness of a pathwise \textit{strong}
solution of~(\ref{dr}) is standard; see, for instance, Theorem 11.2 in
\cite{rogerswilliams}. Furthermore, given the corresponding
assumptions on $b$, if $\zeta(0,\cdot)\in\Omega$, then $\zeta(t,\cdot)\in
\Omega$, for all $t\ge0$.

Using (\ref{X_zeta_eqctd}) with the definition (\ref{etadef}), we
derive by standard It\^{o}-calculus that
%
\begin{equation}
\label{eq:ito}\quad
d \eta(t,x)=
\eta'(t,x)\,dB(t)+\eta'(t,x)\eta(t,0)\,dt+ \frac{\eta''(t,x)}{2} \,dt-b'(x)\,dt.
\end{equation}

We show in   Theorem \ref{thm:stat_erg} that the unique
\textit{Gaussian probability measure}~$\pi(d\omega)$ on $\Omega$ with
mean and covariance
%
\begin{equation}
\label{cov}
\int_{\Omega}\omega(x)\pi(d\omega)=0,\qquad
\int_{\Omega} \omega(x)\omega(y)\pi(d\omega)
=b(x-y)
\end{equation}
is invariant for the \textit{Markov process} $t\mapsto\eta(t):=\eta
(t,\cdot)$.

Recall that Minlos' theorem (Theorem I.10 of \cite{simon74}) implies,
given $x\mapsto b(x)$ with the assumed properties, that the
expectations and covariances (\ref{cov}) define a unique translation
invariant Gaussian probability measure $\pi(d\omega)$ on the space of
tempered distributions $\cS'(\RR)$. The regularity properties of the
covariance function $b$ imply that this measure is actually supported
by the space $\Omega\subset\cS'(\RR)$; see
\mbox{\cite{marcus,marcusrosen}}.

A natural realization of the measure $\pi(d\omega)$ is the following: let
$c\dvtx\RR\to\RR$ be the \textit{unique} function of positive type for
which $b=c*c$ and let $w^\prime(y)$ be standard white noise on the
line $\RR$. Let
%
\begin{equation}
\label{withwhitenoise}
\omega(x):=\int_{\RR} c(x-y)w^\prime(y)\,dy.
\end{equation}
Then the random element
$\omega(\cdot)\in\Omega$ will have exactly the distribution $\pi
(d\omega)$.

Note that the group of spatial translations
%
\begin{equation}
\label{shift}
\RR\ni z\mapsto\tau_z\dvtx\Omega\to\Omega,\qquad
(\tau_z\omega)(x):=\omega(x+z)
\end{equation}
acts naturally on $\Omega$ and preserves the probability measure $\pi
(d\omega)$. As can be seen from the representation (\ref
{withwhitenoise}), the dynamical system $(\Omega, \pi(d\omega), \tau
_z\dvtx z\in\RR)$ is ergodic.
\begin{theorem}
\label{thm:stat_erg}
The Gaussian probability measure $\pi(d\omega)$ on $\Omega$, with
mean~$0$ and covariances (\ref{cov}), is time-invariant and ergodic
for the $\Omega$-valued Markov process $t\mapsto\eta(t)$.\vadjust{\goodbreak}
\end{theorem}

Theorem \ref{thm:stat_erg} is proved in Section \ref{ss:infgen}; we
also provide in Section \ref{thm:fproof} a short formal proof of it.

Now define the function $\varphi\dvtx\Omega\to\RR$,
%
\begin{equation}
\label{phidef}
\varphi(\omega):=\omega(0).
\end{equation}
Note that (\ref{zetaeq}), (\ref{Xeq}) and (\ref{etadef}) imply
%
\begin{equation}
\label{Xint}
X(t)-X(0)= B(t) + \int_0^t\varphi(\eta(s))\,ds.
\end{equation}
The law of large numbers is therefore a direct consequence of ergodicity.
\begin{corollary}
\label{cor:lln}
For $\pi$-almost all initial profiles $\zeta(0,\cdot)$,
%
\begin{equation}
\label{lln2}
\lim_{t\to\infty}\frac{X(t)}{t}=0\qquad
\mbox{a.s.}
\end{equation}
\end{corollary}

This partially settles Conjecture 2 of \cite{durrettrogers92}.

\subsection{Diffusivity bounds on $X(t)$}
All results in the sequel will be meant for the process being in the
stationary regime described in the last section [i.e., $\xi=\zeta
(0,\cdot)\in\Omega$ distributed according to $\pi$].

We now study the $t\to\infty$ asymptotics of the variance of displacement
%
\begin{equation}
\label{diffusivity}
E(t):= \mathbf{E}(X(t)^2).
\end{equation}

First, we use a special kind of time-reversal symmetry, sometimes
called Yaglom-reversibility (see \cite
{yaglom47,yaglom49,dobrushinsuhovfritz88}), to show in Section \ref
{ss:diffusive_lower_bound} that, under the general assumptions of
Section \ref{sec:assumptions}, for any $s<t$, the random variables
$B(t)-B(s)$ and $\int_s^t\varphi(\eta(u))\,du$ are \textit
{uncorrelated}, and, hence,
%
\begin{equation}
\mathbf{E}\bigl(\bigl(X(t)-X(s)\bigr)^2\bigr)=
t-s+\mathbf{E}\biggl(\biggl(\int_s^t \varphi(\eta(u))\,du\biggr)^2\biggr).
\end{equation}

Furthermore, if the following summability condition holds:
%
\begin{equation}
\label{sumcond}
\rho^2:=
\int_{-\infty}^{\infty}p^{-2}\hat b(p)\,dp<\infty,
\end{equation}
then the process $X(t)$ behaves diffusively, as stated in
Theorem \ref{thm:diffusive_bounds}, shown in Section~\ref
{ss:diffusive_upper_bound}. Note that (\ref{sumcond}) is a condition
on the infrared ($|p|\ll1$) asymptotics of the spectrum $\hat b(p)$.
\begin{theorem}
\label{thm:diffusive_bounds}
Let $\rho^2$ be the constant defined in (\ref{sumcond}). Then
%
\begin{equation}
\label{diffusive_bounds}
1
\le
\varliminf_{t\to\infty}t^{-1}E(t)
\le
\varlimsup_{t\to\infty}t^{-1}E(t)
\le
1+\rho^2.
\end{equation}
\end{theorem}
\begin{remarks*}
(1)
The upper bound in (\ref{diffusive_bounds}) is informative only when
the integral on the right-hand side of (\ref{sumcond}) is finite,
which does not hold, for instance, in the self-repelling case $f=-b'$
of the form (\ref{sr}), where $\hat b(0)>0$.\vadjust{\goodbreak}

(2)
This result is short of proving the full CLT, namely, that
\[
\sigma^{2}
:= \lim_{t\to\infty} t^{-1}E(t)
\]
exists, is between the bounds
given in (\ref{diffusive_bounds}) and $t^{-1/2}X(t)\Rightarrow
N(0,\sigma^2)$. Recall that in the follow-up paper
\cite{horvathtothveto09} full CLT is proved for the locally
self-repelling Brownian polymer in $d\ge3$. The proof relies on the
nonreversible Kipnis--Varadhan theory. As explained in that paper,
technical parts of that method cannot be applied (so far) in lower dimensions.

Let, for all $\lambda>0$,
%
\begin{equation}
\hat E(\lambda):=\int_0^\infty e^{-\lambda t} E(t) \,dt,
\end{equation}
and let $D$ be the \textit{diffusivity}, as usually defined: $D(t):=t^{-1}E(t)$.

One can easily show (by a simple change of variables) that, for $\nu>0$,
%
\begin{equation}
\label{Abel_Tauber}
\{
E(t)\sim Ct^{2\nu}, t\gg1
\}
\Rightarrow
\{
\hat E(\lambda)\sim C^{\prime}\lambda^{-2\nu-1}, \lambda\ll1
\}.
\end{equation}
 Theorem \ref{thm:super_diffusive_bounds} shows bounds
for the Laplace transform $\hat E(\lambda)$ as $\lambda\to0$, based
on the \textit{resolvent method}, first used by Landim, Quastel,
Salmhofer and Yau in
\cite{landimquastelsalmhoferyau04} to provide superdiffusive estimates
on the diffusivity of asymmetric simple exclusion process in one and
two dimensions.

Then Lemma \ref{lemma:tauber}, shown in a different context in \cite
{quastelvalko} but readily translated for our purposes (see also
\cite{landimyau,kipnislandim,quastelvalko}), enables us to convert the
\textit{upper} bound on $\hat E(\lambda)$ into an upper bound on
$E(t)$, without the need of extra regularity assumption, as is usually
required in Tauberian theorems. Its proof relies on the estimate of the
variance of additive functionals of Markov processes using the $H_{-1}$ norm.

More\vspace*{1pt} precisely, let us consider the following \textit{infrared bounds}
for the correlation function $\hat b(p)$: for some $-1<\alpha<1$:
%
\begin{equation}
\label{irbounds}
C_1:=\varlimsup_{p\to0} |p|^{-\alpha}\hat b(p)<\infty,\qquad
C_2:=\varliminf_{p\to0} |p|^{-\alpha}\hat b(p)>0.
\end{equation}
Of course, $C_2\le C_1$.
\end{remarks*}
\begin{theorem}
\label{thm:super_diffusive_bounds}
If for some $-1<\alpha<1$ the infrared bounds (\ref{irbounds}) hold, then
%
\begin{equation}
\label{lambda_upper_bound}
\varlimsup_{\lambda\to0}
\lambda^{(5-\alpha)/2} \hat E(\lambda)
\le C_3<\infty
\end{equation}
and
%
\begin{equation}
\label{lambda_lower_bound}
\varliminf_{\lambda\to0}
\lambda^{(9-2\alpha+\alpha^2)/4} \hat E(\lambda)
\ge
C_4>0,
\end{equation}
where the constants $C_3$ and $C_4$ depend only on $\alpha$, $C_1$ and $C_2$.
\end{theorem}
\begin{lemma}
\label{lemma:tauber}
There exists an explicit finite constant $C$ such that
%
\begin{equation}
E(t)\le C t^{-1} \hat E (t^{-1}).\vadjust{\goodbreak}
\end{equation}
\end{lemma}
\begin{remarks*}
(1)
By Lemma \ref{lemma:tauber} the bound (\ref{lambda_upper_bound}) can
be converted into
%
\begin{equation}
\label{t_upper_bound}
\varlimsup_{t\to\infty}
t^{-(3-\alpha)/2} E(t)
\le
C'_3<\infty.
\end{equation}

(2)
Although we cannot translate the lower bound on $\hat E(\lambda)$ into
an asymptotic lower bound on $E(t)$, by (\ref{Abel_Tauber}) the bound
(\ref{lambda_lower_bound}) essentially means
%
\begin{equation}
\label{t_lower_bound}
\varliminf_{t\to\infty}
t^{-(5-2\alpha+\alpha^2)/4} E(t)
\ge
C'_4>0.
\end{equation}

(3)
The locally self-avoiding case corresponds to $\alpha=0$. In this case
our results give
%
\begin{equation}
C''_4 t^{5/4} \le E(t) \le C''_3 t^{3/2}
\end{equation}
with some constants $C''_4>0$, $C''_3<\infty$. Here the first
inequality is meant in the sense of Laplace transforms. Recall that in
this particular case, the conjectured order in
\cite{tothwerner98} is $E(t)\asymp t^{4/3}$.

(4) We make the following conjecture:
\begin{conjecture}
\label{conj:super_diffusive_asymptotics}
Under the conditions of Theorem \ref{thm:super_diffusive_bounds} the
true asymptotic order is
%
\begin{equation}
\label{conj_order}
E(t)=\mathbf{E}(X(t)^2) \asymp t^{{4}/({3+\alpha})}.
\end{equation}
\end{conjecture}
\begin{remark*}
This conjecture is formally in agreement with the order of the limit
proved in \cite{mountfordtarres08} under superballistic scaling, for
slowly decaying (with distance) self-interaction functions $f$, and the
corresponding conjectures formulated in
\mbox{\cite{durrettrogers92,tothwerner98}}.
\end{remark*}
\end{remarks*}

\subsection{\texorpdfstring{Formal proof of Theorem \protect\ref{thm:stat_erg}}{Formal proof of Theorem 1}}
\label{thm:fproof}
In order to prove that $\pi$ is indeed time-stationary, we have to
show that for any (sufficiently smooth) test function~$u(\cdot)$ the
moment generating functional
$
\mathbf{E}(\exp\{\langle u,\eta(t)\rangle\})
$
is constant in time. Here we used the notation
%
\begin{equation}
\langle u, v\rangle:= \int_{-\infty}^{\infty} v(x)u(x)\,dx.
\end{equation}
[Note that starting from Section \ref{section:spaces} the brackets
$\langle\cdot, \cdot\rangle$ will have a different meaning; see
(\ref
{Vscprod}).] It follows from (\ref{eq:ito}) that
%
\begin{eqnarray}\label{condchar}
&&
d\mathbf{E}(\exp\{\langle u,\eta(t)\rangle\})\nonumber\\
&&\qquad= \mathbf
{E}
(d \exp\{\langle u,\eta(t)\rangle\})
\\
&&\qquad= \mathbf{E}\bigl(e^{\langle u,\eta(t)\rangle} \bigl( \tfrac12
\langle
u^{\prime\prime}, \eta(t)\rangle+ \tfrac12 \langle u^{\prime},
\eta(t)\rangle
^2 - \langle u^{\prime}, \eta(t) \rangle\eta(t,0) + \langle
u^{\prime}, b\rangle
\bigr)\bigr)\,dt.\hspace*{-15pt}\nonumber
\end{eqnarray}

Let $X,Y,Z$ be jointly Gaussian with zero mean. Then it is easy to show
(by differentiations of the moment generating function of their\vadjust{\goodbreak} joint
distribution) that
%
\begin{eqnarray}
\mathbf{E}(Ye^X)&=&\exp\{\mathbf{E}(X^2)/2\}
\mathbf{E}(XY),
\\
\mathbf{E}(YZe^X)&=&\exp\{\mathbf{E}(X^2)/2\}
\bigl(\mathbf{E}(YZ)+\mathbf{E}(XY)\mathbf{E}
(XZ)\bigr).
\end{eqnarray}
Using these identities,
if $\eta$ is a zero mean Gaussian field with covariance $b$ (as it is
assumed), the right-hand side of (\ref{condchar}) can be computed
explicitly to deduce
%
\begin{eqnarray}
&&e^{(1/2)\langle u, b* u\rangle}
\bigl\{
\tfrac12
\langle u^{\prime\prime}, b*u\rangle
+
\tfrac12
\langle u^{\prime}, b*u^{\prime}\rangle\nonumber\\[-8pt]\\[-8pt]
&&\hphantom{e^{(1/2)\langle u, b* u\rangle}
\bigl\{}
{} +
\tfrac12
\langle u^{\prime}, b*u\rangle^2
-
\langle u^{\prime}, b*u\rangle\langle u, b\rangle
\bigr\}\,dt.\nonumber
\end{eqnarray}
Note that for any test function $u$ we have
$\langle u^{\prime}, b*u\rangle=0$, since $b$ is even.
Thus, after one integration by parts we note that
the previous expression is always 0, which shows that
$\mathbf{E}(\exp\{\langle u,\eta(t)\rangle\})$
is indeed constant in time.
\begin{remark*}
It is not hard to check that translation invariant Gaussian fields with
nonzero centering and the same covariances
%
\begin{equation}
\label{cov2}
\int_{\Omega}\omega(x)\pi(d\omega)=v\in\RR,\qquad
\int_{\Omega} \omega(x)\omega(y)\pi(d\omega) -v^2
=b(x-y)
\end{equation}
are also time-stationary (and ergodic) for the process $t\mapsto\eta
(t)$. If we start our process with these initial distributions, then
the corresponding laws of large numbers
%
\begin{equation}
\label{lln3}
\lim_{t\to\infty}\frac{X(t)}{t}=v
\qquad\mbox{a.s.}
\end{equation}
hold, which means \textit{ballistic behavior} of the process $t\mapsto
X(t)$. We will not pursue these regimes in the present note.
\end{remark*}

\section{Spaces and operators}
\label{section:spaces}

\subsection{Spaces}
\label{ss:measures}

The natural formalism for the proofs of our theorems is that of Fock
space and Gaussian Hilbert spaces. We follow the usual notation of
Euclidean quantum field theory; see, for example,
\cite{simon74}.

Endow the space of real-valued smooth functions of rapid decrease\break
(Schwartz space) $\cS=\cS(\RR)$ with the inner product
%
\begin{eqnarray}
\label{Vscprod}
\langle u,v\rangle:\!&=&
\int_{-\infty}^\infty\int_{-\infty}^\infty u(x)b(x-y)v(y)\,dx\,dy\nonumber\\[-8pt]\\[-8pt]
&=&
\int_{-\infty}^\infty\hat u(-p) \hat v(p) \hat b(p)\,dp
<\infty,\nonumber
\end{eqnarray}
and let $\cV$ be the completion of $\cS(\RR)$ with respect to this
Euclidean norm.

We denote $\cH:=\cL^2(\Omega,\pi)$. Then
%
\begin{equation}
\label{Vimbedding}
\phi\dvtx\cS\to\cH,\qquad
\phi(v)(\omega):= \int_{-\infty}^\infty\omega(x)v(x)\,dx\vadjust{\goodbreak}
\end{equation}
is an isometric embedding of $(\cV, \langle\cdot,\cdot\rangle)$ in
$\cH$:
%
\begin{equation}
\label{isometry}
\Vert\phi(v)\Vert^2_{\cH}
=\Vert v \Vert^2_{\cV},
\end{equation}
so $\phi$ extends as an isometric embedding of $\cV$ into the
Gaussian subspace of~$\cH$.

The Hilbert space $\cH$ is naturally graded
%
\begin{equation}
\label{Hgrading}
\cH=\cH_0\oplus\cH_1\oplus\cH_2\oplus\cdots\oplus\cH_n\oplus
\cdots,
\end{equation}
where
%
\begin{eqnarray}
\label{grade0}
\cH_0&:=&\{c\one, c\in\RR\},
\\
\label{grade1}
\cH_1&:=&\{\phi(v), v\in\cV\},
\\
\label{graden}
\cH_n&:=&\operatorname{span}
\{\wick{\phi(v_1)\cdots\phi(v_n)}, v_1,\ldots,v_n\in\cV\}.
\end{eqnarray}
Here and throughout the rest of the paper $\wick{X_1\cdots X_n}$
denotes the \textit{Wick product} of the jointly Gaussian random
variables $(X_1,\ldots, X_n)$. For basics of Fock space and Wick
products see, for example, chapter I of
\cite{simon74} and/or chapter III of \cite{janson97}.

\subsection{Operators}
\label{ss:operators}

We use the standard notation of Fock spaces.
Given a~(bounded or unbounded) closed linear operator $A$ over the
basic Hilbert space~$\cV$, its second quantized version over the
Hilbert space $\cH$ will be denoted $d\Gamma(A)$. This latter one
acts on Wick monomials as follows:
%
\begin{equation}
\label{sqop}\quad
d\Gamma(A){\wick{\phi(v_1)\cdots\phi(v_j)\cdots\phi(v_n)}}
=
\sum_{j=1}^n
\wick{\phi(v_1)\cdots\phi(Av_j)\cdots\phi(v_n)},
\end{equation}
and it is extended by linearity and graph closure.

A particularly important linear operator over $\cV$ is the
differentiation with respect to the $x$-variable:
%
\begin{equation}
\label{grad}
\partial v(x) := v^\prime(x).
\end{equation}
This is an unbounded skew self-adjoint operator defined on the dense domain
%
\begin{equation}
\label{domgrad}
\operatorname{Dom}(\partial)=
\biggl\{v\in\cV\dvtx\int_{-\infty}^\infty p^2|\hat v(p)|^2\hat b(p)\,dp<\infty
\biggr\}.
\end{equation}
We denote the second quantization of $\partial$ by
%
\begin{equation}
\label{sqgrad}
\nabla:=d\Gamma(\partial).
\end{equation}
Then $\nabla$ is also an unbounded and skew self-adjoint operator over
$\cH$. We shall also need the operator $\nabla^2$ acting on $\cH$.
(Note that this is \textit{not} the second quantization of~$\partial^2$.)

Given an element $u\in\cV$, the creation and annihilation (or:
raising and lowering) operators associated to it are
%
\begin{equation}
\label{cran1}
a^*(u)\dvtx\cH_n\to\cH_{n+1},\qquad
a(u)\dvtx\cH_n\to\cH_{n-1},
\end{equation}
acting on Wick monomials as
%
\begin{eqnarray}
\label{cropdef}
a^*(u)\wick{\phi(v_1)\cdots\phi(v_n)}
&=&
\wick{\phi(u)\phi(v_1)\cdots\phi(v_n)},
\\
\label{anopdef}
a(u)\wick{\phi(v_1)\cdots\phi(v_n)}
&=&
\sum_{j=1}^n\langle u,v_j\rangle
{\wick{\phi(v_1)\cdots\phi(v_{j-1})\phi(v_{j+1})\cdots\phi(v_n)}}\mbox{.}\hspace*{-28pt}
\end{eqnarray}
We will also use the following straightforward commutation relation:
%
\begin{equation}\label{ccr2}
[\nabla,a(u)]=a(u').
\end{equation}
We define the unitary involution $J$ on $\cH$:
%
\begin{equation}
\label{invop}
Jf(\omega):=f(-\omega),\qquad
J\upharpoonright_{\cH_n}=(-1)^n I\upharpoonright_{\cH_n}.
\end{equation}
The subspace of smooth functions
%
\begin{equation}
\label{core}\quad
\cC:=\{
F(\phi(v_1),\ldots,\phi(v_k))\dvtx F\in C_0^\infty(\RR^k\to\RR),
v_1,\ldots,v_k\in\cS
\}\subset\cH
\end{equation}
is a \textit{common core} for all (unbounded) operators defined above
and used in the sequel. They act on functions of this form as follows:
%
\begin{eqnarray}\qquad\quad
\label{nablaonsmooth}
\nabla F(\phi(v_1),\ldots,\phi(v_k))
&=&
\sum_{l=1}^k \partial_l
F(\phi(v_1), \ldots, \phi(v_k))
\phi(v_l^\prime),
\\
\label{nablasquaredonsmooth}
\nabla^2 F(\phi(v_1),\ldots,\phi(v_k))
&=&
\sum_{l,m=1}^k \partial^2_{l,m}
F(\phi(v_1), \ldots, \phi(v_k))
\phi(v_l^\prime)\phi(v_m^\prime)
\nonumber\\[-8pt]\\[-8pt]
&&{}
+
\sum_{l=1}^k \partial_l
F(\phi(v_1), \ldots, \phi(v_k))
\phi(v_l^{\prime\prime}),
\nonumber\\
\label{anoponsmooth}
a(u)F(\phi(v_1),\ldots,\phi(v_k))
&=&
\sum_{l=1}^k \partial_l
F(\phi(v_1), \ldots, \phi(v_k))
\langle u,v_l\rangle,
\\
\label{croponsmooth}
a^*(u)F(\phi(v_1),\ldots,\phi(v_k))
&=&
\phi(u)F(\phi(v_1), \ldots, \phi(v_k))\nonumber\\[-8pt]\\[-8pt]
&&{}-
a(u)F(\phi(v_1),\ldots,\phi(v_k)).\nonumber
\end{eqnarray}
For basics about creation, annihilation and second quantized operators
see, for example,
\cite{janson97} or \cite{simon74}. In particular, note that, for
$F\in\cC$ and $u\in\cV$ such that
\mbox{$b*u\in\Omega$}, the following identities hold:
%
\begin{eqnarray}
\label{multipl}
\bigl(\bigl(a^*(u)+a(u)\bigr)F\bigr)(\omega)
&=&
(\phi(u)F)(\omega),
\\
\label{dirder}
(a(u)F)(\omega)
&=&
\lim_{\vareps\to0}
\vareps^{-1}\bigl(F(\omega+\vareps b*u)-F(\omega)\bigr).
\end{eqnarray}
These identities are easily checked on Wick monomials and extended by
linearity. Identity (\ref{multipl}) means that the sum of the creation
and annihilation operators corresponding to an element of the basic
space $\cV$ is the multiplication operator with the isometric Gaussian
embedding of that vector. The meaning of (\ref{dirder}) is that the
annihilation operator $a(u)$ is actually a~``directional derivative'' in
the direction $b*u\in\Omega$. This latter one is a particular case of
a well-known identity from Malliavin calculus; see, for example,
chapter~XV and, in particular, Theorem 15.8 of~\cite{janson97}.

Notice that $\nabla$ is the infinitesimal generator of the \textit
{unitary group of spatial translations}
while $\nabla^2/2$ is the infinitesimal generator of the Markovian
semigroup of \textit{diffusion in random scenery}:
%
\begin{eqnarray}
\label{shiftgroup}
\exp\{z\nabla\}&=&T_z,\qquad
T_zf(\omega):=f(\tau_z\omega),
\\
\label{drscesemigroup}
\exp\{t\nabla^2/2\}&=&Q_t,\qquad
Q_tf(\omega):=\int\frac{\exp\{-z^2/(2t)\}}{\sqrt{2\pi t}} f(\tau
_z\omega)\,dz.
\end{eqnarray}

\subsection{The infinitesimal generator, stationarity,
Yaglom-reversibility, ergodicity}
\label{ss:infgen}

We denote
%
\begin{equation}
\label{semigroup}
P_t\dvtx\cH\to\cH,\qquad
P_t f(\omega)
:=
\mathbf{E}\bigl(f(\eta(t))|\eta(0)=\omega\bigr).
\end{equation}
Then $[0,\infty)\ni t\mapsto P_t\in\cB(\cH)$
(bounded operators on $\cH$)
is a positivity
preserving contraction semigroup on~$\cH$.

Given $f=F(\phi(v_1),\ldots, \phi(v_k))\in\cC$, from (\ref
{X_zeta_eqctd}), (\ref{etadef}) and using (\ref
{nablaonsmooth})--(\ref{croponsmooth}), one can compute
%
\begin{equation}
\label{infgencomp}\quad
\lim_{t\to0}\frac{\mathbf{E}(f(\eta(t)-f(\eta(0)))
|\eta(0)=\omega)}{t}
=
\biggl(\frac12 \nabla^2 + \phi(\delta) \nabla+ a(\delta^\prime)
\biggr)f(\omega).
\end{equation}

This operator is extended from $\cC$ by graph closure. Now, using the
commutation relation (\ref{ccr2}), we obtain the \textit{infinitesimal
generator} of the semigroup~$P_t$:
%
\begin{equation}
\label{infgen}
G:=\tfrac12 \nabla^2 + a^*(\delta) \nabla+ \nabla a(\delta).
\end{equation}
The adjoint of the generator is
%
\begin{equation}
\label{adjinfgen}
G^*:=\tfrac12 \nabla^2 - a^*(\delta) \nabla- \nabla a(\delta).
\end{equation}
For later use we introduce notation for the symmetric (self-adjoint)
and antisymmetric (skew-self-adjoint) parts of the generator:
%
\begin{eqnarray}
\label{symgen}
S&:=&-\tfrac12(G+G^*)=-\tfrac12 \nabla^2,
\\
A&:=&\tfrac12(G-G^*)= a^*(\delta)\nabla+\nabla a(\delta)
=:A_++A_-.
\end{eqnarray}
Note that
%
\begin{eqnarray}
\label{grading}
S\dvtx\cH_n&\to&\cH_n,\qquad
A_+\dvtx\cH_n\to\cH_{n+1},\nonumber\\[-8pt]\\[-8pt]
A_-\dvtx\cH_n&\to&\cH_{n-1},\qquad
A_-=-A_+^*\nonumber
\end{eqnarray}
and
%
\begin{equation}
\label{van_H0_H1}
S\upharpoonright_{\cH_0}=0,\qquad
A_+\upharpoonright_{\cH_0}=0,\qquad
A_-\upharpoonright_{\cH_0\oplus\cH_1}=0.
\end{equation}
\begin{pf*}{Proof of Theorem \ref{thm:stat_erg} and Corollary
\ref{cor:lln}}
It is clear that
%
\begin{equation}
\label{stateq}
G^*\one= 0,
\end{equation}
and, hence, it follows that $\pi$ is indeed a stationary distribution
of the process $t\mapsto\eta(t)$ and $G^*$ is itself the
infinitesimal generator of the stochastic semigroup $P^*_t$ of the time
reversed process.

Proving ergodicity is easy. For any $f\in\cH$ the Dirichlet form of
the process $t\mapsto\eta(t)$ is given by
%
\begin{equation}
\label{df}
\cD(f):=
-(f,G f)=
-\bigl(f, \tfrac12\nabla^2 f\bigr)=
\tfrac12\| {\nabla f} \|^2,
\end{equation}
where $(\cdot, \cdot)$ and \mbox{$\| \cdot \|$} denote the
scalar product
and $L_2$ norm in $\cH$.
So,
%
\begin{equation}
\label{erg}
\{ \cD(f)=0 \}
\quad\Leftrightarrow\quad
\{ \nabla f=0 \}
\quad\Leftrightarrow\quad
\{ f=\mbox{const. } \pi\mbox{-a.s.} \},
\end{equation}
since $z\mapsto\tau_z$ acts ergodically on $(\Omega,\pi)$.

Corollary \ref{cor:lln} follows directly (\ref{Xint}), by the ergodic theorem.
\end{pf*}

The generator $G$ is, of course, not reversible, but the so-called
\textit{Yaglom-reversibility}
\cite{yaglom47,yaglom49,dobrushinsuhovfritz88} holds:
%
\begin{equation}
\label{yaglom}
G^{*}=JGJ.
\end{equation}
This identity means that the stationary forward process $(-\infty
,\infty)\ni t\mapsto\eta(t)$ and 
%
\begin{equation}
\label{revproc}
(-\infty,\infty)\ni t\mapsto\tilde\eta(t):=-\eta(-t)
\end{equation}
obey the same law. We will call $t\mapsto\tilde\eta(t)$ the \textit
{flipped-backward process}.

\section{Diffusive bounds}
\label{section:diffusive_bounds}


The aim of this section is to prove Theorem \ref{thm:diffusive_bounds}.

\subsection{Diffusive lower bound}
\label{ss:diffusive_lower_bound}

For $-\infty<s\le t<\infty$ denote
%
\begin{equation}
\label{marti}
M(s,t)
:=
X(t)-X(s)-\int_s^t\varphi(\eta(u))\,du
=
B(t)-B(s).
\end{equation}
\begin{lemma}
\label{lemma:forwbackw}
For $s\in\RR$ fixed the process $[s,\infty)\ni t\mapsto M(s,t)$ is a
forward martingale with respect to the \textit{forward filtration} $\{
\cF_{(-\infty,t]}\dvtx t\ge s\}$ of the process $t\mapsto\eta(t)$. For
$t\in\RR$ fixed the process $(-\infty,t]\ni s\mapsto M(s,t)$ is a
backward martingale with respect to the \textit{backward filtration} $\{
\cF_{[s,\infty)}\dvtx s\le t\}$ of the process $t\mapsto\eta(t)$.
\end{lemma}
\begin{pf}
There is nothing to prove about the first statement: the integral on
the right-hand side of (\ref{marti}) was chosen exactly so that it
compensates the conditional expectation of the infinitesimal increments
of $X(t)$.\vadjust{\goodbreak}

We turn to the second statement of the lemma.
We use the following facts:

%
%
%

(1) For any $s\le t$, there is a Borel function $F_{s,t}$ mapping a.s.
$(\eta(u))_{s\le u\le t}$ to $X(t)-X(s)$.
By symmetry, $F_{-t,-s}$ maps the flipped-backward process $(\tilde
\eta(u))_{-t\le u\le -s}$ in (\ref{revproc}) to
%
\begin{equation}
\tilde X(-s) - \tilde X(-t) = X(s)-X(t).
\end{equation}

(2)
The forward process $t\mapsto\eta(t)$ and flipped-backward process
$t\mapsto\tilde\eta(t)$ are  identical in law.

(3)
The function $\omega\mapsto\varphi(\omega)$ is odd with respect to the
flip map $\omega\mapsto -\omega$.

Putting these facts together (in this order) we obtain
%
\begin{eqnarray}
\label{bwmart}
&&
\lim_{h\to0}\mathbf{E}\biggl(\frac{X(s-h)-X(s)}{-h}\bigg|\cF_{[s,\infty)}\biggr)\nonumber\\
&&\qquad=
-\lim_{h\to0}\mathbf{E}\biggl(\frac{\tilde X(-s+h)-\tilde X(-s)}{h}\bigg|\tilde
\cF_{(-\infty,-s]}\biggr)\\
&&\qquad=
-\varphi(\tilde\eta(-s))
=\varphi(\eta(s)).\nonumber
\end{eqnarray}
\upqed\end{pf}

From Lemma \ref{lemma:forwbackw} it follows that
%
\begin{eqnarray}
\label{variancesum}
\mathbf{E}\bigl(\bigl(X(t)-X(s)\bigr)^2\bigr)
&=&
\mathbf{E}\bigl(\bigl(M_{[s,t]}\bigr)^2\bigr)+
\mathbf{E}\biggl(\biggl(\int_s^t \varphi(\eta(u))\,du\biggr)^2\biggr)
\nonumber\\[-8pt]\\[-8pt]
&=&
t-s
+
\mathbf{E}\biggl(\biggl(\int_s^t \varphi(\eta(u))\,du\biggr)^2
\biggr),\nonumber
\end{eqnarray}
hence the lower bound in (\ref{diffusive_bounds}).

\subsection{Diffusive upper bound}
\label{ss:diffusive_upper_bound}

Throughout this section we assume (\ref{sumcond}).
First we recall a general result about the limiting variance of
additive functionals integrated along the trajectory of a stationary
and ergodic Markov process.

Let $t\mapsto\eta(t)$ be a stationary and ergodic Markov process on
the abstract probability space $(\Omega, \pi)$. Denote the
infinitesimal generator acting on $\cL^2(\Omega,\pi)$ and its
adjoint by $G$, respectively, by $G^*$. These might be unbounded
operators, but it is assumed that they have a common core of
definition. Denote the symmetric (self-adjoint), respectively, the
antisymmetric (skew-self-adjoint) part of the infinitesimal generator by
%
\begin{equation}
\label{symmasymmgengen}
S:=-\tfrac12(G+G^*),\qquad
A:=\tfrac12(G-G^*).
\end{equation}
Let $t\mapsto\xi(t)$ be the \textit{reversible} Markov process on the
same state space $(\Omega, \pi)$ which has the infinitesimal
generator $-S$.

The following lemma is proved in
\cite{sethuramanvaradhanyau00}. See also the survey papers \cite
{olla,landimollavaradhan} and further references cited therein.
\begin{lemma}
\label{lemma:v}
Let $\varphi\in\cL^2(\Omega,\pi)$ with $\int\varphi \,d\pi=0$.
Then
%
\begin{equation}
\label{svy_bound}
\qquad\varlimsup_{t\to\infty} t^{-1}\mathbf{E}\biggl(\biggl(\int
_0^t\varphi(\eta(s))\,ds\biggr)^2\biggr)\le
\lim_{t\to\infty} t^{-1}\mathbf{E}\biggl(\biggl(\int_0^t\varphi
(\xi(s))\,ds\biggr)^2\biggr).
\end{equation}
\end{lemma}

In our particular case
%
\begin{equation}
\label{drsgen}
S=-\tfrac12\nabla^2,
\end{equation}
and the reversible process
$t\mapsto\xi(t)$ will be the so-called \textit{diffusion in random
scenery} process; see, for example,
\cite{kestenspitzer79} or the more recent survey \cite
{denhollandersteif06}. That means
%
\begin{equation}
\label{rwrs}
\xi(t):=\tau_{Z_t}\omega,
\end{equation}
where $t\mapsto Z_t$ is a standard Brownian motion, independent of the
field $\omega$. The function $\varphi\dvtx\Omega\to\RR$ is $\varphi
(\omega)=\omega(0)$. Thus, the upper bound in (\ref{svy_bound}) will be
%
\begin{eqnarray}
\label{ourbound}
\lim_{t\to\infty} t^{-1}\mathbf{E}\biggl(\biggl(\int_0^t\varphi
(\xi(s))\,ds\biggr)^2\biggr)
&=&
\lim_{t\to\infty} t^{-1}\mathbf{E}\biggl(\biggl(\int_0^t\omega
(Z_s)\,ds\biggr)^2\biggr)\nonumber\\[-8pt]\\[-8pt]
&=&
\int_{-\infty}^\infty p^{-2} \hat b(p)\,dp.\nonumber
\end{eqnarray}
Here the last step is just explicit computation, with expectation taken
over the Brownian motion $Z(t)$ \textit{and} over the random scenery
$\omega$. The straightforward details are left for the reader.

\section{Superdiffusive bounds}
\label{section:superdiffusive_bounds}

From (\ref{variancesum}) it follows that
%
\begin{eqnarray}
E(t)&=&
t +
\mathbf{E}\biggl(\biggl(\int_0^t\varphi(\eta(s))\,ds\biggr)^2\biggr)\nonumber\\[-8pt]\\[-8pt]
&=&
t +
2\int_{0}^t (t-s)\mathbf{E}(\varphi(\eta(s))\varphi(\eta
(0))) \,ds.\nonumber
\end{eqnarray}

Taking the Laplace transform of the previous equation, we get
%
\begin{equation}
\label{restoD}
\hat E(\lambda)=
\lambda^{-2}
\bigl(1+ 2 \bigl(\varphi, (\lambda-G)^{-1}\varphi\bigr)\bigr).
\end{equation}
We will estimate $(\varphi, (\lambda-G)^{-1}\varphi)$
using the following variational formula; see, for example, (2.5) of
\cite{landimquastelsalmhoferyau04}:
%
\begin{eqnarray}
\label{varform}
&&\bigl(\varphi, (\lambda-G)^{-1}\varphi\bigr)
\nonumber\\[-8pt]\\[-8pt]
&&\qquad=
\sup_{\psi\in\cH}\bigl\{
2(\varphi, \psi)
-
\bigl(\psi, (\lambda+S)\psi\bigr)
-
\bigl(A\psi, (\lambda+S)^{-1} A\psi\bigr)
\bigr\}.\nonumber
\end{eqnarray}

\subsection{Superdiffusive upper bounds}
\label{ss:superdiffusive_upper_bounds}

\mbox{}

\begin{pf*}{Proof of Theorem \ref{thm:super_diffusive_bounds}---upper
bound}
The {upper bound} will follow from simply dropping the last term on the
right-hand side of (\ref{varform}):
%
\begin{equation}
\label{varform_upper}\qquad\quad
\bigl(\varphi, (\lambda-G)^{-1}\varphi\bigr)
\le
\sup_{\psi\in\cH}\bigl\{
2(\varphi, \psi)
-
\bigl(\psi, (\lambda+S)\psi\bigr)
\bigr\}
=
\bigl(\varphi, (\lambda+S)^{-1}\varphi\bigr).
\end{equation}
Note that---modulo a Tauberian inversion---this is equivalent to the
argument in Section \ref{ss:diffusive_upper_bound}.

Using (\ref{drscesemigroup}) and (\ref{symgen}), we write the
resolvent of $-S$ as
%
\begin{equation}
\label{Sresolvent}\quad
(\lambda+S)^{-1}=
\int_{-\infty}^{\infty} \int_0^\infty
\frac{1} {\sqrt{2\pi t}} e^{-\lambda t-z^2/(2t)} T_z
\,dt \,dz=
\int_{-\infty}^{\infty}
g_{\lambda}(z) T_z \,dz,
\end{equation}
where the function $g_{\lambda}(z)$ and its Fourier transform $\hat
g_{\lambda}(p)$ are
%
\begin{equation}
g_{\lambda}(z)
=\frac{1} {\sqrt{2\lambda}} e^{-\sqrt{2\lambda}|z|},\qquad
\hat g_{\lambda}(p)
=\frac{1}{\sqrt{2\pi}} \frac{1}{\lambda+ p^2/2}.
\end{equation}
Hence, by the Parseval formula,
\begin{eqnarray*}
\bigl(\varphi, (\lambda+S)^{-1}\varphi\bigr)
&=&
\int_{-\infty}^\infty g_\lambda(z)\mathbf{E}(\omega(0)\omega
(z))\\
&=&
\frac{1}{\sqrt{2\pi}}
\int_{-\infty}^{\infty} \frac{\hat b(p)} {\lambda+ p^2/2} \,dp.
\end{eqnarray*}
By (\ref{irbounds}), we can choose $\delta>0$ so that for $|p|<\delta$
%
\begin{equation}
\label{choosedelta}
\frac{C_2}{2} |p|^{\alpha}\le\hat b(p)\le2 C_1 |p|^{\alpha}.
\end{equation}
Then
%
\begin{eqnarray}
\label{upper}
&&\bigl(\varphi, (\lambda+S)^{-1}\varphi\bigr)\nonumber\\
&&\qquad\le
C_1\sqrt{\frac{2}{\pi}}
\int_{-\infty}^\infty\frac{|p|^{\alpha}}{\lambda+ p^2/2}\,dp
+
\sqrt{\frac{2}{\pi}}
\int_{|p|>\delta} p^{-2} \hat b(p)\,dp
\\
&&\qquad=
\lambda^{({\alpha-1})/{2}}
C_1 \sqrt{\frac{2}{\pi}}
\int_{-\infty}^{\infty} \frac{|q|^{\alpha}}{1+q^2/2}\,dq +
\sqrt{\frac{2}{\pi}}
\int_{|p|>\delta} p^{-2} \hat b(p)\,dp.\nonumber
\end{eqnarray}
Since both integrals in (\ref{upper}) are finite (as $|\alpha|<1$),
the upper bound (\ref{lambda_upper_bound}) follows from (\ref
{restoD}), (\ref{varform_upper}) and (\ref{upper}).
\end{pf*}

\subsection{Superdiffusive lower bounds}
\label{ss:superdiffusive_lower_bounds}

\mbox{}

\begin{pf*}{Proof of Theorem \ref{thm:super_diffusive_bounds}---lower
bound}
Lower bounds are obtained by taking on the right-hand side of (\ref
{varform}) the supremum over\vadjust{\goodbreak} the subspace~$\cH_1$ only:
%
\begin{eqnarray}
&&\bigl(\varphi, (\lambda-G)^{-1}\varphi\bigr)\nonumber\\[-8pt]\\[-8pt]
&&\qquad\ge
\sup_{\psi\in\cH_1}\bigl\{
2(\varphi, \psi)
-
\bigl(\psi, (\lambda+S)\psi\bigr)
-
\bigl(A\psi, (\lambda+S)^{-1} A\psi\bigr)
\bigr\}\nonumber
\\
\label{varform_1}
&&\qquad=
\sup_{\psi\in\cH_1}\bigl\{
2(\varphi, \psi)
-
\bigl(\psi, (\lambda+S)\psi\bigr)
-
\bigl(A_+\psi, (\lambda+S)^{-1} A_+\psi\bigr)
\bigr\}.
\end{eqnarray}
The last identity is due to (\ref{van_H0_H1}).

We write $\psi\in\cH_1$ as
%
\begin{equation}
\psi= \int_{-\infty}^\infty u(x)\omega(x)\,dx
\end{equation}
with $u$ an even function and compute the three terms on the right-hand
side of (\ref{varform_1}). The first two are straightforward:
%
\begin{eqnarray}
\label{first}
\hspace*{-20pt}(\varphi, \psi)
&=&
\int_{-\infty}^\infty u(x)\mathbf{E}(\omega(0)\omega(x))\,dx
=
\int_{-\infty}^{\infty}
\hat b(p) \hat u(p) \,dp,\\[-18pt]\nonumber
\end{eqnarray}
\begin{eqnarray}
\label{second}\hspace*{50pt}
\hspace*{-20pt}&&\bigl(\psi, (\lambda+S)\psi\bigr)\nonumber\\
&&\qquad=
\int_{-\infty}^\infty
\int_{-\infty}^\infty
\biggl(
\lambda u(x)u(y) + \frac12 u^{\prime}(x)u^{\prime}(y)
\biggr)
\mathbf{E}(\omega(x)\omega(y))
\,dx\,dy\\
\hspace*{-20pt}&&\qquad=
\int_{-\infty}^{\infty}
(\lambda+ p^2/2)
\hat b(p) \hat u(p)^2 \,dp.\hspace*{-35pt}\nonumber
\end{eqnarray}
In order to compute the third term, we first note that
%
\begin{equation}
\label{A+psi}
A_+\psi=
\int_{-\infty}^\infty u^{\prime}(x) \wick{\omega(0)\omega(x)} dx
\end{equation}
and, hence,
%
\begin{eqnarray}
\label{third2}\hspace*{5pt}
&&
\bigl(A_+\psi, (\lambda+S)^{-1} A_+\psi\bigr)
\nonumber\\
&&\qquad
=
\int_{-\infty}^\infty
\int_{-\infty}^\infty
\int_{-\infty}^\infty
g_\lambda(z)u^{\prime}(x)u^{\prime}(y)
\mathbf{E}\bigl(\wick{\omega(0)\omega(x)}\,\wick{\omega(z)\omega
(z+y)}\bigr)
\,dx\,dy\,dz
\nonumber\\
&&\qquad
=
\int_{-\infty}^\infty
\int_{-\infty}^\infty
\int_{-\infty}^\infty
g_\lambda(z)u^{\prime}(x)u^{\prime}(y)
\bigl(b(z)b(z+y-x)\nonumber\\[-8pt]\\[-8pt]
&&\qquad\quad\hspace*{139pt}{} + b(z+y)b(z-x)\bigr)
\,dx\,dy\,dz
\nonumber\\
&&\qquad
=
\frac{1}{2\sqrt{2\pi}}
\int_{-\infty}^\infty
\int_{-\infty}^\infty
\frac{\hat b(p)\hat b(q)}{\lambda+ (p-q)^2/2}
\bigl(
p \hat u(p)
-
q \hat u(q)
\bigr)^2
\,dq\,dp
\nonumber\\
&&\qquad
\le
\int_{-\infty}^\infty
\hat b(p)
p^2 \hat u(p)^2K(\lambda,p)\,dp,\nonumber
\end{eqnarray}
where
%
\begin{eqnarray}
K(\lambda,p):\!&=&
\frac{1}{\sqrt{2\pi}}
\int_{-\infty}^\infty
\frac{\hat b(q)}{\lambda+ (p+q)^2/2}
\,dq\nonumber\\[-8pt]\\[-8pt]
& =&
\frac{1}{\sqrt{2\pi}}
\int_{-\infty}^\infty
\frac{\hat b(q-p)}{\lambda+ q^2/2}
\,dq.\nonumber
\end{eqnarray}
In the last step we used the Cauchy--Schwarz inequality and the fact
that~$\hat b(\cdot)$ is a nonnegative even function.

From (\ref{varform_1}), (\ref{first}), (\ref{second}) and (\ref
{third2}) it follows that
%
\begin{equation}
\label{lower}
\bigl(\varphi, (\lambda-G)^{-1}\varphi\bigr)
\ge
\int_{-\infty}^{\infty}
\frac{\hat b(p)}
{\lambda+ p^2/2 + K(\lambda,p) p^2}
\,dp.
\end{equation}
Next we give an \textit{upper} bound for $K(\lambda,p)$. Let $\delta$
be chosen so that the bounds~(\ref{choosedelta}) hold and assume that
$\lambda<\delta^2/4$. Then
%
\begin{eqnarray}\hspace*{5pt}
&&{\mathbh1}_{\{|p|<\lambda^{1/2}\}}K(\lambda, p)\nonumber\\
&&\qquad\le
C_1\sqrt{\frac{2}{\pi}}
{\mathbh1}_{\{|p|<\lambda^{1/2}\}}
\int_{-\infty}^{\infty}
\frac{|q-p|^{\alpha}}{\lambda+ q^2/2}\,dq
+
\sqrt{\frac{2}{\pi}}
\int_{|q|>\delta/2} q^{-2} \hat b(q-p) \,dq
\nonumber\\
\label{intermed}
&&\qquad\le
\lambda^{(\alpha-1)/2}
C_1\sqrt{\frac{2}{\pi}}
\sup_{|r|<1}
\int_{-\infty}^{\infty}
\frac{|q-r|^{\alpha}}{1+ q^2/2}\,dq
+
\sqrt{\frac{2}{\pi}}
\int_{|q|>\delta/2} q^{-2} \hat b(q-p) \,dq
\\
\label{Kupper}
&&\qquad\le
C\lambda^{(\alpha-1)/2}
\end{eqnarray}
with some $C<\infty$, for $\lambda$ sufficiently small. The last
inequality holds since the integrals in (\ref{intermed}) are bounded.

From (\ref{lower}) and (\ref{Kupper}) it follows that for
sufficiently small $\lambda$
%
\begin{eqnarray}
\label{lower2}
\bigl(\varphi, (\lambda-G)^{-1}\varphi\bigr)
&\ge&
\int_{|p|\le\lambda^{1/2} }
\frac{\hat b(p)}
{\lambda+ C \lambda^{(\alpha-1)/2}p^2}
\,dp
\nonumber\\
&\ge&
\frac{C_2}{2}
\int_{|p|\le\lambda^{1/2} }
\frac{|p|^{\alpha}}
{\lambda+ C \lambda^{(\alpha-1)/2}p^2}\,dp
\nonumber\\[-8pt]\\[-8pt]
&=&
\lambda^{-(1-\alpha)^2/4}
\frac{C_2}{2}
\int_{|r|\le\lambda^{(\alpha-1)/4}}
\frac{|r|^{\alpha}}
{1 + C r^2}\,dr
\nonumber\\
&\ge&
C\lambda^{-(1-\alpha)^2/4}\nonumber
\end{eqnarray}
with some $C>0$, for $\lambda$ sufficiently small.

The lower bound (\ref{lambda_lower_bound}) follows from (\ref
{restoD}) and (\ref{lower2}).
\end{pf*}

\section*{Acknowledgments}
P. Tarr\`{e}s thanks A.-S. Sznitman for a very stimulating discussion.
B. T\'{o}th thanks   the Mittag Leffler Insitute,\vadjust{\goodbreak}
Stockholm, for their kind hospitality, where part of this work was done. B. Valk\'{o} thanks
J.~Quastel for
many enlightening conversations.


%

%
\printaddresses

\end{document}